\pdfoutput=1
\RequirePackage{ifpdf}
\ifpdf 
\documentclass[pdftex]{sigma}
\else
\documentclass{sigma}
\fi

\numberwithin{equation}{section}

\newtheorem{Theorem}{Theorem}[section]
\newtheorem{Proposition}[Theorem]{Proposition}
 { \theoremstyle{definition}
\newtheorem{Definition}[Theorem]{Definition}
\newtheorem{Example}[Theorem]{Example}
\newtheorem{Remark}[Theorem]{Remark} }

\begin{document}
\allowdisplaybreaks

\newcommand{\arXivNumber}{1907.02925}

\renewcommand{\PaperNumber}{065}

\FirstPageHeading

\ShortArticleName{Solvable Lie Algebras of Vector Fields and a Lie's Conjecture}

\ArticleName{Solvable Lie Algebras of Vector Fields\\ and a Lie's Conjecture}

\Author{Katarzyna GRABOWSKA~$^\dag$ and Janusz GRABOWSKI~$^\ddag$}

\AuthorNameForHeading{K.~Grabowska and J.~Grabowski}

\Address{$^\dag$~Faculty of Physics, University of Warsaw, Poland}
\EmailD{\href{mailto:katarzyna.grabowska@fuw.edu.pl}{katarzyna.grabowska@fuw.edu.pl}}

\Address{$^\ddag$~Institute of Mathematics, Polish Academy of Sciences, Poland}
\EmailD{\href{mailto:jagrab@impan.pl}{jagrab@impan.pl}}

\ArticleDates{Received February 04, 2020, in final form July 02, 2020; Published online July 10, 2020}

\Abstract{We present a local and constructive differential geometric description of finite-dimensional solvable and transitive Lie algebras of vector fields. We show that it implies a Lie's conjecture for such Lie algebras. Also infinite-dimensional analytical solvable and transitive Lie algebras of vector fields whose derivative ideal is nilpotent can be adapted to this scheme.}


\Classification{17B30; 17B66; 57R25; 57S20}

\section{Introduction}

Sophus Lie began to study the so-called continuous groups around 1870. He was led to these objects by two motives. One was to extend to differential equations the results of Galois theory for algebraic equations; the second to study the transformation groups of general geometrical structures, in particular, to study and extend the properties of the group of contact transformations and ``canonical'' transformations, i.e., groups that arose in differential geometry and classical mechanics. We should remark, however, that the
objects that Lie studied are not groups in the present day sense of the word. They are families of diffeomorphisms, locally defined on a~manifold, closed under the operations of composition and inverse whenever these are defined. These objects are today called pseudogroups.

One instance of these pseudogroups were the ``finite continuous groups'' which can be regarded as representations by diffeomorphisms of abstract finite dimensional Lie groups (or restrictions thereof). The theory of these groups was studied in great detail by Killing, Cartan, Weyl, etc., and the theory of Lie groups is now a standard part of the mathematical curriculum.

In the early days of Lie theory, the problem of classifying Lie algebras of vector
fields under local diffeomorphisms played a central role in the subject, notably
because of the applications to~the integration of differential equations admitting
infinitesimal symmetries. Lie himself classified the Lie algebras of vector fields in
one real variable, one complex variable and two complex variables \cite{Lie2,Lie1} (see also \cite{92}).
We should however note that Lie, like all the geometers
of the nineteenth century, dealt only with the analytic case and did not
consider the $C^\infty$ case.

In Lie theory and representation theory, the \emph{Levi decomposition}, conjectured by Killing and Cartan and proved by Eugenio Elia Levi (1905), states that any finite-dimensional real Lie algebra $\mathfrak{g}$ is the semidirect product of a solvable ideal and a semisimple subalgebra. One is its \emph{radical}, a maximal solvable ideal, and the other is a semisimple subalgebra, called a \emph{Levi subalgebra}. The Levi decomposition implies that any finite-dimensional Lie algebra is a semidirect product of a~solvable Lie algebra and a~semisimple Lie algebra, it is enough to classify separately semisimple and solvable Lie algebras.

Every semisimple Lie algebra over an algebraically closed field of characteristic 0 is a direct sum of simple Lie algebras (by definition), and the finite-dimensional simple Lie algebras fall in four families: $A_n$, $B_n$, $C_n$, and $D_n$, with five exceptions $E_6$, $E_7$, $E_8$, $F_4$, and $G_2$. Simple Lie algebras are classified by the connected Dynkin diagrams, while semisimple Lie algebras correspond to not necessarily connected Dynkin diagrams, where each component of the diagram corresponds to a summand of the decomposition of the semisimple Lie algebra into simple Lie algebras. Semisimple Lie algebras over the field of complex numbers ${\mathbb C}$
have been completely classified by E.~Cartan \cite{Cartan1894}, over the field of
real numbers ${\mathbb R}$ by F.~Gantmacher \cite{Gantmacher39}.

To classify all finite-dimensional Lie algebras it would be therefore sufficient to classify solvable algebras. Unfortunately, the problem of classification of all solvable (including nilpotent) Lie algebras in
an arbitrarily large finite dimension is presently unsolved and is generally believed to be unsolvable.
All known full classifications terminate at relatively low dimensions, e.g., the classification of nilpotent
algebras is available at most in dimension eight \cite{10,11}, for the solvable ones in dimension six \cite{12,13}.
A direct construction of realizations of low-dimensional Lie algebras from their structure constants can be found in \cite{Popovych03}. Projectable Lie algebras of vector fields in (complex) dimension three i.e., Lie algebras on ${\mathbb C}^2\times {\mathbb C}$ which project on ${\mathbb C}^2$ were studied recently in \cite{Schneider18}.

Realizations of Lie algebras in vector fields are especially interesting. The interest in the local classification problem for the Lie algebras of vector fields comes from applications in the general theory
of differential equations, Lie systems, in group classification of ODEs and PDEs, in geometric control theory, in the theory of systems with superposition principles, etc.
Some results on classification of transitive local realizations, together with powerful methods
of explicit computation are already present in the literature \cite{Blattner69,Draisma12,GS64,SW014}.
A lot of information, especially about the solvable Lie algebras of vector fields, one can find in survey \cite{SW014} and the references therein. Transitive solvable Lie algebras of vector fields have been also studied in the context of integrability by quadratures \cite{CFGR15,CFG16}. A purely algebraic and very instructive approach that relates realizations of a Lie algebra with its subalgebras has been proposed in \cite{Magazev15}.

As dealing with vector fields on manifolds, we will work exclusively with Lie algebras over reals.
Probably the easiest way of constructing a nilpotent Lie algebra of vector fields is to consider vector fields with negative degree with respect to a positive dilation on ${\mathbb R}^n$. To be more precise, let us associate with the coordinates $\big(x^1,\dots,x^n\big)$ positive integer degrees (weights) $(w_1,\dots,w_n)$. This is equivalent to picking up
the \emph{weight vector field}
\[
\nabla^h=\sum_iw_ix^i\partial_{x^i} ,
\] or a one-parametric family of positive \emph{dilations} (called in \cite{Bruce:16, Grabowski:2006,Grabowski:2012}, a \emph{homogeneity structure})
\[
h_t\big(x^1,\dots,x^n\big)=\big(t^{w_1}x^1,\dots,t^{w_n}x^n\big) .
\]
The biggest $w_i$ is called the \emph{degree} of the homogeneity structure. The weight vector field defines \emph{homogeneous functions of degree~$k$} as those $f\in C^\infty\big({\mathbb R}^n\big)$
for which $\nabla^h(f)=kf$. One can prove~\cite{Grabowski:2012} that only non-negative homogeneity degrees are allowed and that each homogeneous function is polynomial in $\big(x^1,\dots,x^n\big)$. This can be extended to tensors, e.g., a~vector field $X$ is homogeneous of degree $k$ (what we denote by $X\in\mathfrak{g}^k(h)$) if $\big[\nabla^h,X\big]=kX$. This time, $k$ may be negative, for instance $\deg\big(\partial_{x^i}\big)=-w_i$ while $\deg\big(x^j\partial_{x^i}\big)=w_j-w_i$.
Since $\deg(X)\ge -w(h)=\min\{-w_1,\dots,-w_n\}$ and, for homogenous $X$, $Y$,
\[ \deg([X,Y])=\deg(X)+\deg(Y) ,\]
the family of vector fields of negative degrees
\[ \mathfrak{g}^{<0}(h)=\bigoplus_{i=-w(h)}^{-1}\mathfrak{g}^i(h)\]
is a finite-dimensional nilpotent Lie algebra (cf.~\cite{JG90,MK88}).
This (except for finite-dimensionality) remains valid for general graded supermanifolds~\cite{Voronov10}, in particular graded bundles in the sense of~\cite{Grabowski:2012}.

It is an interesting observation that any transitive nilpotent Lie algebra $L$ of vector fields defined on a neighbourhood of $0\in{\mathbb R}^n$ is locally a subalgebra of $\mathfrak{g}^{<0}(h)$ for some dilation $h$ on~${\mathbb R}^n$ \cite{JG90} (cf.\ also~\cite{MK88}, where the assumptions are a little bit stronger). The nilpotent algebra~$L$ is therefore automatically finite-dimensional and polynomial in appropriate homogeneous coordinates $\big(x^1,\dots,x^n\big)$. The homogeneity structure is uniquely defined by the structure of $L$, so that $\mathfrak{g}^{<0}(h)$ can be viewed as nilpotent `enveloping' or `prolongation' of $L$. Note also that every finite-dimensional (real) Lie algebra is a transitive Lie algebra of vector fields as the Lie algebra of left-invariant vector fields on the corresponding 1-connected Lie group.

The homogeneity structures defined in \cite{Grabowski:2012} are a little bit more general: they admit coordinates of degree $0$, so that the manifold is a fibration over the quotient manifold corresponding to zero-degree coordinates (graded bundle). In particular, graded bundles of degree 1 are just vector bundles.
The family of vector fields $\mathfrak{g}^{<0}(h)$ is still nilpotent, but it is a subalgebra in the family of vector fields with non-positive degrees,
\[ \mathfrak{g}^{\le 0}(h)=\bigoplus_{i=-w(h)}^{0}\mathfrak{g}^i(h) .\]
The latter is generally not solvable, but any finite-dimensional Lie algebra $L$ in $\mathfrak{g}^{\le 0}(h)$ such that $[L,L]\subset\mathfrak{g}^{< 0}(h)$ is solvable.
This is because, for finite-dimensional Lie algebras, $L$ is solvable if and only if $[L,L]$ is nilpotent (this is not true in infinite dimensions). In other words the \emph{nilradical} $\mathfrak{nr}(L)$ of $L$ contains $[L,L]$. Solvable Lie algebras of vector fields in $\mathfrak{g}^{\le 0}(h)$ we will call \emph{dilational} (with respect to~$h$).

If we distinguish coordinates $x^1,\dots,x^p$ of zero degree (weight) and coordinates of positive degree (weight) $z^1_{w_1},\dots,z^q_{w_q}$, where $w_i>0$ indicates the weight, then locally the homogeneity structure (dilation)
looks like (see \cite{Grabowski:2012})
\[ h_t(x,z)=\big(x^1,\dots,x^p,t^{w_1}z^1_{w_1},\dots,t^{w_q}z^q_{w_q}\big)\]
and smooth functions of degree $\le w$ are polynomials with coefficients in functions depending on~$x$:
\[ f(x,y)=\sum_{|\alpha|\le w}f_\alpha(x)\big(z_{w_1}^1\big)^{\alpha_1}\cdots \big(z_{w_q}^q\big)^{\alpha_q},
\]
where for $\alpha\in{\mathbb N}^q$, we put $|\alpha|=\sum_iw_i\alpha_i$. The method of defining dilational transitive solvable Lie algebras of vector fields is illustrated by the following example.

\begin{Example}\label{1.1}
On ${\mathbb R}^4$ with local coordinates $(x,y,z,u)$ we define $h$ declaring that $x$ is of weight~0, $y$ is of weight 1, $z$ is of weight~2, $u$ is of weight~3. The nilpotent part $\mathfrak{g}^{< 0}(h)$ is a module over ${\mathbb A}=C^\infty({\mathbb R})=\{ f(x)\}$ generated by polynomial vector fields
\begin{gather*} \mathfrak{g}^{-3}(h)={\mathbb A}[ \partial_u ], \qquad \mathfrak{g}^{-2}(h)={\mathbb A}[ y\partial_u, \partial_z ], \qquad
\mathfrak{g}^{-1}(h)={\mathbb A}\big[ z\partial_u, y^2\partial_u, y\partial_z, \partial_y \big].
\end{gather*}
We can now choose a commutative algebra of vector fields of weight 0 completing the above to a transitive algebra,
e.g.,
\[ \partial_x, \ y\partial_y, \ z\partial_z, \ u\partial_u ,\]
 and choose dilational solvable transitive Lie algebras of vector fields as transitive
subalgebras of
\begin{gather}\label{ida}\langle \partial_x, y\partial_y, z\partial_z, u\partial_u\rangle\oplus{\mathbb A}\left[ \partial_u, y\partial_u, \partial_z, \partial_u, z\partial_u, y^2\partial_u, y\partial_z, \partial_y \right].
\end{gather}
\end{Example}

Our aim in this note is to show that any transitive solvable finite-dimensional Lie algebra~$L$ of vector fields, defined locally in a~neighbourhood of $0\in{\mathbb R}^n$, is dilational, i.e., it is a~subalgebra of~$\mathfrak{g}^{\le 0}(h)$, with $[L,L]\subset\mathfrak{g}^{< 0}(h)$, for a homogeneity structure $h$ on ${\mathbb R}^n$ and appropriate homogeneous coordinates.
This is of course far from fully classifying solvable Lie algebras of vector fields, but it gives a nice structural information about $L$. This also allows for an easy way of constructing solvable Lie algebras of vector fields. Note also that our methods are constructive and serve as well for a large variety of infinite-dimensional Lie algebras of vector fields, especially in the analytical context.

We prove also the Lie's conjecture for solvable Lie algebras. The conjecture states that for any finite-dimensional complex transitive Lie algebra $L$ of vector fields in ${\mathbb R}^n$ one can find coordinates $y^1,\dots ,y^n$ in which all coefficients of vector fields from $L$ lie in the algebra generated by the~$y^i$ and the exponentials $\exp\big(\lambda y^i\big)$.

\section{Nilpotent Lie algebras of vector fields}
Our main result slightly depends on the fact whether $L$ itself is nilpotent or not. If yes, then we recover the result proved already in~\cite{JG90}. We will reformulate this result as follows.

The following is well known. Let $L$ be a Lie algebra. We define the lower central series of $L$ inductively:
\[ L=L^1 ,\qquad L^{i+1} = \big[L,L^i\big] \qquad \text{for}\quad i\ge 1 .\]
\begin{Definition} If the lower central series \emph{stops at zero}, i.e.,
$L^{k+1}=\{ 0\}$ for some $k$, then we say that~$L$ is \emph{nilpotent}.
The smallest such $k$ we call the \emph{height} of the nilpotent Lie algebra.
\end{Definition}
Let us recall from the introduction that on ${\mathbb R}^n$ we consider a positive dilation, i.e., an action~$h$ of the multiplicative monoid $({\mathbb R},\cdot)$
of the form
\[ h_t\big(x^1,\dots,x^n\big)=\big(t^{w_1}x^1,\dots,t^{w_n}x^n\big) ,\qquad t\in{\mathbb R} ,\]
where $w_i\in{\mathbb Z}$, $w_i>0$, $w(h)=\max\{w_1,\ldots,w_n\}.$ The associated \emph{weight vector field} reads $\nabla^h=\sum_iw_ix^i\partial_{x^i}$.
A smooth function $f$ defined in a neighbourhood of $0\in{\mathbb R}^n$ is homogeneous of degree $a$ if
$\nabla^h(f)=af$. Note that only non-negative integer homogeneity degrees are possible.
Similarly, a vector field $X$ is homogeneous of degree $a$ if
$\big[\nabla^h,X\big]=aX$, however now negative degrees not less than $-w(h)$ are allowed.
The family of homogeneous vector fields (of degree~$a$) is denoted by~$\mathfrak{g}^a(h)$.

We observe that the family of vector fields of negative degrees
\[ \mathfrak{g}^{<0}(h)=\bigoplus_{i=-w(h)}^{-1}\mathfrak{g}^i(h)\]
is a transitive (finite-dimensional) nilpotent Lie algebra of vector fields (cf.~\cite{JG90,MK88}).
The transitivity means that the vector fields of negative degrees span ${\mathsf T} {\mathbb R}^n$.
Of course, any transitive subalgebra of $\mathfrak{g}^{<0}(h)$ is again transitive nilpotent Lie algebra of vector fields. We can, of course, restrict $\mathfrak{g}^{<0}(h)$ to a neighbourhood $U$ of $0\in{\mathbb R}^n$ which gives a Lie algebra $\mathfrak{g}^{<0}(U,h)$ isomorphic to $\mathfrak{g}^{<0}(h)$. In this case we will speak about a \emph{local Lie algebra of vector fields}.
\begin{Example}
Consider the dilation $h$ on ${\mathbb R}^2=\{(y,z)\}$ such that $y$ is of degree 1 and $z$ is of degree 2. The Lie algebra of vector fields with negative degrees is spanned by $\langle\partial_y,\partial_z,y\partial_z\rangle$ which is transitive nilpotent of height 2 since $L^2=[L,L]=\langle \partial_z\rangle$.
\end{Example}

For a nilpotent local Lie algebra $L$ of vector fields on ${\mathbb R}^n$ of height $k$ we consider the descending series of finite-dimensional vector spaces $L^i(0)\subset{\mathsf T}_0{\mathbb R}^n$, spanned by values at $0\in{\mathbb R}^n$ of vector fields belonging to $L^i$,
\begin{gather*}
L(0)=L^1(0)={\mathsf T}_0{\mathbb R}^n\supset L^2(0)\supset \cdots\supset L^{k}(0)\supset L^{k+1}(0)=\{ 0\} .
\end{gather*}
Let us rewrite this sequence, collecting terms of the same dimension of the space and indicating this dimension by the lower indices:
\begin{gather*}
 \underbrace{L^1_n(0)=\cdots= L^{r_1}_n(0)}_{r_1}\supsetneq
\underbrace{L^{r_1+1}_{n-b_1}(0)=\cdots= L^{r_2}_{n-b_1}(0)}_{r_2-r_1}\supsetneq
\cdots\supsetneq\\
\qquad{} \supsetneq\underbrace{L^{r_{m-1}+1}_{n-b_{m-1}}(0)=\cdots= L^{r_m}_{n-b_{m-1}}(0)}_{r_m-r_{m-1}}
\supsetneq \{ 0\} .
\end{gather*}
Put $r_0=0$, $b_0=0$, $a_0=0$, $b_m=n$, and $a_i=b_i-b_{i-1}$ for $i>0$. Of course, $a_1+\cdots+a_m=n$.
The main result in \cite{JG90} can be reformulated as follows.
\begin{Theorem}\label{nil}
Let $L$ be a transitive local nilpotent Lie algebra of vector fields on ${\mathbb R}^n$ and let~$r_i$,~$b_i$ and~$a_i$, $i=0,\dots,m$, be defined as above. There exist local variables $\big(y^1,\dots,y^n\big)$ in a~neighbourhood $U$ of $0\in{\mathbb R}^n$ such that $y^1,\dots, y^{b_j}$ vanish on $L^{r_j+1}$ and $\partial_{y^{b_j+1}},\dots,\partial_{y^{b_{j+1}}}$ span $L^{r_{j}+1}(0)/L^{r_{j+1}}(0)$, $j=1,\dots,m-1$. Let $h$ be a positive dilation of ${\mathbb R}^n$ such that the first $a_1$ variables, $y^1,\dots,y^{a_1}$, are of degree $r_1>0$, the next $a_2$ variables are of degree~$r_2$, etc., up to final~$a_m$ variables, $y^{b_{m-1}+1},\dots,y^n$ which are of degree $r_{m}$.
Then, with respect to this dilation, $L\subset\mathfrak{g}^{<0}(U,h)$ and
\[ L^j\subset \bigoplus_{i=-k}^{-j}\mathfrak{g}^i(U,h) ,\]
where $k$ denotes the height of~$L$.
In particular, the algebra $L$ is finite-dimensional and polynomial in the chosen coordinates.
\end{Theorem}
\begin{Remark}
Note that the dilation above, so $\mathfrak{g}^{<0}(U,h)$, is completely determined (up to a~diffeomorphism) by the dimensions of $L^i(0)$, so the algebra of vector fields $L$.
\end{Remark}
\begin{Example}
Consider the Lie algebra of vector fields on ${\mathbb R}^2$ spanned by
$L=L^1=\big\langle\partial_x,\partial_y,y\partial_x,\allowbreak y^2\partial_x\big\rangle$. It is nilpotent of height 3: $L^2=\langle\partial_x,y\partial_x\rangle$, $L^3=\langle\partial_x\rangle$. According to the theorem, we associate with $y$ degree 1 and with $x$ degree 3, so $L$ is spanned by vector fields of degree $\le -1$.
\end{Example}
\begin{Remark} Let us now consider nilpotent Lie algebra $L$ included in
\[ \mathfrak{g}^{<0}(h_0)=\bigoplus_{i=-w(h_0)}^{-1}\mathfrak{g}^i(h_0)\]
for a given dilation $h_0$. We observe that the dilation $h$ determined by $L$ according to Theorem~\ref{nil} need not to recover $h_0$, so in general $h\ne h_0$. For instance if coordinate $x$ is of degree 0 and $y$ is of degree $2$ for $h_0$, for the commutative algebra $L=\langle\partial_x,\partial_y\rangle$ it is sufficient to take $h$ for which~$x$ and~$y$ are of degree~1.
\end{Remark}

\section{Solvable algebras: descending series of ideals and foliations}
Let $L$ be a Lie algebra. We define the \emph{nilradical series} of $L$ inductively:
\begin{gather}\label{hcs} L^0=L ,\qquad L^1=\mathfrak{nr}(L) , \qquad L^{i+1}=\big[\mathfrak{nr}(L),L^i\big]\qquad \text{for}\quad i\ge 1 ,
\end{gather}
where $\mathfrak{nr}(L)$ is the nilradical of $L$.
Its name will be justified by the fact that it its main part (i.e., starting from $i=1$) is the lower central series of the nilradical of $L$. The following observation is fundamental. Note a shift in the notation in comparison with the nilpotent case due to the fact that we will have to accept coordinates of degree~$0$.
\begin{Proposition}
For a finite-dimensional $L$ $($over a field of characteristic~$0)$, the nilradical series is central, i.e., $[L,L]\subset\mathfrak{nr}(L)$, if and only if~$L$ is solvable.
\end{Proposition}
\begin{proof}
It is easy to see that the nilradical series is a central series if and only if $[L,L]$ is nilpotent. For finite-dimensional Lie algebras (over a field of characteristic~$0$) this is equivalent to solvability.
\end{proof}
\begin{Remark}
As the main result we are approaching is known to be valid for nilpotent Lie algebras of vector fields, we can in general assume that $L^0=L\ne L^1=\mathfrak{nr}(L)$. It is also obvious that the series (\ref{hcs}) stops at zero, i.e., $L^{k+1}=\{ 0\}$ for some non-negative integer~$k$.
\end{Remark}
\begin{Example} The proposition is not true in infinite-dimension. For instance, consider the Lie algebra $L$ spanned by elements $\langle x,a,h_1,h_2,\dots\rangle$ subject to the Lie bracket
\[ [x,a]=a ,\qquad [x,h_i]=ih_i ,\qquad [a,h_i]=h_{i+1} ,\qquad [h_i,h_j]=0 .\]
The derived ideal $[L,L]$ is spanned by $\langle a,h_1,h_2,\dots\rangle$ and it is solvable but not nilpotent.
\end{Example}

Now, we will realize a (real) finite-dimensional solvable Lie algebra $L$ as a \emph{local transitive Lie algebra of vector fields on ${\mathbb R}^n$}, i.e., a Lie algebra of smooth vector fields defined in a neighbourhood $U$ of $0\in{\mathbb R}^n$ such that the vector fields from $L$ span the whole tangent space ${\mathsf T}_0{\mathbb R}^n$, i.e.,
\[ L(0)={\mathsf T}_0{\mathbb R}^n\simeq{\mathbb R}^n .\]
We can of course equivalently assume that $L$ spans the tangent spaces ${\mathsf T}_x{\mathbb R}^n$ for all $x\in U$.
As the generalized distributions $V^i$ generated in $U$ by $L^i$ are invariant with respect to the infinitesimal action of~$L$, they are invariant with respect to the local Lie group corresponding to $L$ which transitively acts by local diffeomorphisms in a neighbourhood of $0\in{\mathbb R}^n$, say $U$ itself (cf.~\cite{Palais}). The distributions are therefore regular and clearly involutive, so define foliations ${\mathcal F}^i$ on $U$. Thus we get the following.

\begin{Proposition}
If $L$ is a local transitive finite-dimensional solvable Lie algebra of vector fields in a neighbourhood of $0\in{\mathbb R}^n$, then ideals $L^i$ from the nilradical series
\[ L=L^0\vartriangleright L^1\vartriangleright L^2\vartriangleright\cdots\vartriangleright L^k\vartriangleright L^{k+1}=\{ 0\}\]
of $L$ generate a descending series
\begin{gather}\label{fol}{\mathcal F}^0=\{ U\}\supset{\mathcal F}^1\supset{\mathcal F}^2\supset \cdots\supset {\mathcal F}^{k}\supset {\mathcal F}^{k+1}=U \end{gather}
of foliations in a neighbourhood $U$ of $0\in{\mathbb R}^n$.
\end{Proposition}

Inclusion ${\mathcal F}^i\supset{\mathcal F}^{i+1}$ means that leaves of foliation ${\mathcal F}^{i+1}$ are submanifolds of the leaves of foliation ${\mathcal F}^{i}$. Note that the series of foliations~(\ref{fol}) need not be strictly decreasing and that we understand that $L^k\ne\{ 0\}$.
\begin{Example}\label{ex2}
Let $L$ be the solvable Lie algebra of vector fields on ${\mathbb R}^2$ spanned by
\[ L=\big\langle \partial_x ,\partial_y ,x\partial_x ,y\partial_y ,y^2\partial_x ,y\partial_x\big\rangle .\]
Then,
\[ L^1=[L,L]=\big\langle \partial_x ,\partial_y ,y^2\partial_x ,y\partial_x\big\rangle ,\]
so ${\mathcal F}^1={\mathcal F}^0=\big\{{\mathbb R}^2\big\}$. Moreover, $L^2=\big[L^1,L^1\big]=\langle\partial_x ,y\partial_x\rangle$ and $L^3=\big[L^1,L^2\big]=\langle\partial_x\rangle$, so ${\mathcal F}^2={\mathcal F}^3$ are generated by $\partial_x$ and $k=3$.
\end{Example}
\begin{Remark}We consider only transitive Lie algebras, since non-transitive ones are much harder to describe. In particular, in the non-transitive case we can have huge infinite-dimensional commutative Lie algebras of vector fields, as for example
the algebra $L=\langle f(y)\partial_x\rangle$ on ${\mathbb R}^2$. Here, $f$ is an arbitrary smooth function of one variable.
Any full description of such algebras seems extremely difficult.
\end{Remark}

Let us fix a local solvable finite-dimensional algebra of vector fields and pick up from the sequence (\ref{fol}) the subsequence
\begin{gather}\label{fol1}
{\mathcal F}^{r_1}\supsetneq{\mathcal F}^{r_2}\supsetneq \cdots\supsetneq {\mathcal F}^{r_m}\supsetneq {\mathcal F}^{r_{m+1}}=U \end{gather}
consisting of those ${\mathcal F}^i$ for which the dimension of ${\mathcal F}^{i+1}$ is smaller than that of ${\mathcal F}^i$. In particular, ${\mathcal F}^{r_m}={\mathcal F}^k$. Put
\begin{gather}\label{adef}a_i=\dim\big({\mathcal F}^{r_i}\big)-\dim\big({\mathcal F}^{r_i+1}\big) ,\qquad i=1,\dots,m .\end{gather}
Of course, $a_i>0$, $a_m=\dim\big({\mathcal F}^k\big)$, and $a_1+\cdots+a_m=n$, so that $0=b_0<b_1<b_2<\cdots <b_m=n$, where $b_j=a_1+\cdots+a_j=\dim\big({\mathcal F}^{r_j}\big)$ for $j>0$.
Hence, the series~(\ref{fol}) looks like
\begin{gather*}
\underbrace{{\mathcal F}^{0}_n=\cdots={\mathcal F}^{r_1}_n}_{r_1+1}\supsetneq
\underbrace{{\mathcal F}^{r_1+1}_{n-b_1}=\cdots={\mathcal F}^{r_2}_{n-b_1}}_{r_2-r_1}\supsetneq
\cdots\supsetneq
\underbrace{{\mathcal F}^{r{_m-1}+1}_{n-b_{m-1}}=\cdots={\mathcal F}^{r_m}_{n-b_{m-1}}}_{r_m-r_{m-1}}
\supsetneq U ,
\end{gather*}
where with the lower indices we indicated the dimension of the leaves of the foliation. In the above example it is a subsequence ${\mathcal F}^1$, ${\mathcal F}^3$, so that $a_1=a_2=1$.

\begin{Proposition}[cf.~\cite{CFG16, CFGR15}]\label{p1}
There are coordinates $y^1,\dots,y^n$ in a neighbourhood of $0\in{\mathbb R}^n$ such that:
\begin{enumerate}\itemsep=0pt
\item[$1.$] The coordinates $y^1,\dots, y^{b_j}$ define ${\mathcal F}^{r_{j+1}}$ as their level sets for each $j=1,\dots,m$.
\item[$2.$] There are $Y_1,\dots,Y_n\in L$ such that, for each $j=1,\dots,m$, the vector fields $Y_{b_{j-1}+1},\dots,Y_n$ generate ${\mathcal F}^{r_j}$ and
\begin{gather}\label{gen} Y_k=\partial_{y^k} \qquad \operatorname{mod}{\mathcal F}^{r_{j+1}} ,\end{gather}
for $k=b_{j-1}+1,\dots,b_j$.
\item[$3.$] The generators of ${\mathcal F}^k$ are the vector fields
\begin{gather}\label{fine} Y_{b_{m-1}+1}=\partial_{y^{b_{m-1}+1}} ,\qquad \dots,\qquad Y_n=\partial_{y^n}\end{gather}
belonging to $L^k$.
\item[$4.$] Moreover, every $X\in L$ can be written as
\begin{gather}\label{form}X=\sum_{j=b_{m-1}+1}^n\left(\sum_{i=b_{m-1}+1}^n f_i^j\big(y^1,\dots,y^{b_{m-1}}\big)y^i+g^j\big(y^1,\dots,y^{b_{m-1}}\big)\right)\partial_{y^j}
+\bar{X} ,
\end{gather}
where the vector field $\bar{X}$ depends on the first $b_{m-1}$ variables only,
\begin{gather*}
\bar{X}=\sum_{i=1}^{b_{m-1}}v^i\big(y^1,\dots,y^{b_{m-1}}\big)\partial_{y^i} .\end{gather*}
If $X\in L^1$, then
\begin{gather*}
X=\sum_{j=b_{m-1}+1}^nu^j\big(y^1,\dots,y^{b_{m-1}}\big)\partial_{y^j}
+\bar{X} .
\end{gather*}

\item[$5.$] The vector fields $\bar{X}$ form a local Lie algebra $\bar{L}$ of vector fields in a neighbourhood $\bar{U}$
of $0\in{\mathbb R}^{b_{m-1}}$ and $L\ni X\mapsto \bar{X}\in\overline{L}$ is a Lie homomorphism onto $\bar{L}$. The algebra $\bar{L}$ is therefore solvable and transitive finite-dimensional Lie algebra of vector fields in $b_{m-1}=k-a_m$ variables.
\end{enumerate}
\end{Proposition}
\begin{proof} The coordinates $y^1,\ldots, y^n$ will be constructed inductively. In the following $V^i$ denotes the distribution spanned by elements of $L^i$ and $L^i(0)=V^i\cap{\mathsf T}_0 U$. In particular $V^{r_1}={\mathsf T} U$ since~$L$ is transitive.
Take $Y_1,\dots,Y_{b_1}\in L^{r_1}$ such that $Y_1(0),\dots,Y_{b_1}(0)$ generate $L^{r_1}(0)$ $\operatorname{mod} L^{r_2}(0)$. Hence, $Y_1,\dots,Y_{b_1}$ span $V^{r_1}$ $\operatorname{mod} V^{r_2}$ in a neighbourhood of $0\in{\mathbb R}^n$.
Let now $\alpha^1,\dots,\alpha^{b_1}$ be 1-forms in a neighbourhood on $0\in{\mathbb R}^n$, vanishing on $V^{r_2}$ and such that
$\alpha^l(Y_s)=\delta^l_s$, $l,s=1,\dots, b_1$. These 1-forms are closed. Indeed,
\[ {\rm d}\alpha^l(Y_s,Y_{s'})=Y_s\big(\alpha^l(Y_{s'})\big)- Y_{s'}\big(\alpha^l(Y_{s})\big)-\alpha^l([Y_s,Y_{s'}])=0 .\] Note that $\alpha^l([Y_s,Y_{s'}])=0$, since
\[ [Y_s,Y_{s'}]\subset \big[L^{r_1},L^{r_1}\big]\subset L^{r_1+1},\]
so that $[Y_s,Y_{s'}]$ is a section of the distribution $V^{r_1+1}=V^{r_2}$.
Similarly, for $X,X'\in L^{r_2}$ we have
\[ {\rm d}\alpha^l(Y_s,X)=0 \qquad\text{and}\qquad {\rm d}\alpha^l(X,X')=0 .\]
As $L^{r_2}$ spans $V^{r_2}$, this completes the proof that ${\rm d}\alpha^l=0$.
We can therefore (uniquely) choose functions $y^1,\dots,y^{b_1}$ vanishing at $0\in{\mathbb R}^n$ such that $\alpha^l={\rm d} y^l$. In particular, $y^j$ are constant on leaves of ${\mathcal F}^{r_2}$, $Y_s\big(y^l\big)=\delta^l_s$. Note that $u=\big(y^1,\ldots, y^{b_1}\big)$ parameterize leaves of ${\mathcal F}^{r_2}$. By $F(u)$ we will denote the leaf of ${\mathcal F}^{r_2}$ corresponding to $u$.

Similarly, we can choose $Y_{b_1+1},\dots,Y_{b_2}\in L^{r_2}$ which span $V^{r_2} \operatorname{mod} V^{r_3}$ in a neighbourhood of $0\in{\mathbb R}^n$. As before, on each leaf $F(u)$ of ${\mathcal F}^{r_2}$ we can find coordinates $y^{b_1+1}_u,\dots,y_u^{b_2}$ which are constant on the leaves of ${\mathcal F}^{r_3}_{|F(u)}$ and such that, on $F(u)$, we have $Y_s\big(y^l_u\big)=\delta^l_s$. The functions $y^{b_1+1}_u,\dots,y_u^{b_2}$ are determined up to a constant on each leaf. If we choose them so that they vanish on a smooth curve through $0\in{\mathbb R}^n$ and transversal to the leaves of ${\mathcal F}^{r_2}$, they smoothly depend on $u$ and define smooth functions $y^{b_1+1},\dots,y^{b_2}$ in a neighbourhood of $0\in{\mathbb R}^n$.

At this point of the construction we have vector fields $Y_1,\dots,Y_{b_2}$ and smooth functions $y^1,\dots,y^{b_2}$. The functions $\big(y^1,\dots,y^{b_1}\big)$ define the foliation ${\mathcal F}^{r_2}$, and ($y^1,\dots,y^{b_2}$) define ${\mathcal F}^{r_3}$. Moreover, $Y_s\big(y^l\big)=\delta^l_s$ for $s,l=1,\dots,b_1$ and for $s,l=b_1+1,\dots,b_2$.
Proceeding inductively in this way, we finally get vector fields $Y_1,\dots,Y_n\in L$ and coordinates $y^1,\dots,y^n$ in a neighbourhood of $0\in{\mathbb R}^n$ with the desired properties.

Finally, from (\ref{gen}) we deduce (\ref{fine}), i.e., $V^{k}=V^{r_m}$ is generated by $Y_{b_{m-1}+1}=\partial_{y^{b_{m-1}+1}}$ up to $Y_n=\partial_{y^n}$. Hence, every $Y\in L^k$ is of the form
\begin{gather}\label{last} Y=\sum_{j=b_{m-1}+1}^n\gamma_j(y)\partial_{y^j} .\end{gather}
But $L^k$ is commutative, so that $[\partial_{y^i},Y]=0$ for $i=b_{m-1}+1,\dots,n$, that implies that $\gamma_j$ depends on $\big(y^1,\ldots,y^{b_{m-1}}\big)$ only. Since, $L^k$ is an ideal, for every $X\in L$
we have that $[\partial_{y^i},X]$ is of the form~(\ref{last}). This, in turn, implies~(\ref{form}) and \begin{gather*}
\overline{[X,Y]}=[\overline{X},\overline{Y}] .\tag*{\qed}\end{gather*}\renewcommand{\qed}{}
\end{proof}

\begin{Example}\label{ex3}
Let $n=1$, which means we are looking for local finite-dimensional solvable Lie algebras $L$ of vector fields on the real line.
As ${\mathcal F}^k=\{{\mathbb R}\}$, the algebra $L^k$ contains $Y_1=\partial_y$. As there are no other variables, according to (\ref{form}),
each $X\in L$ is locally of the form $X=(ay+b)\partial_y$. Hence, $L=\langle\partial_y\rangle$ or $L=\langle\partial_y ,y\partial_y\rangle$. Thus we recovered the classical result known already to S.~Lie. In the first case of $L=\langle\partial_y\rangle$ the algebra is nilpotent, so in view of Theorem~\ref{nil} we associate the degree~$1$ to the variable~$y$. In the second case of $L=\langle\partial_y\rangle$ which is solvable but not nilpotent we also associate the degree~$1$ to the variable~$y$ in preparation for the next section. We have then~$\partial_y$ of degree~$-1$ and $y\partial_y$ of degree~$0$.
\end{Example}

\section{Solvable Lie algebras are dilational}
\begin{Theorem}\label{main}
Let $L$ be a finite-dimensional transitive local solvable Lie algebra of vector fields on ${\mathbb R}^n$ and let $r_i$ and $a_i$, $i=0,\dots,m$, be defined as in \eqref{fol1} and \eqref{adef}. Then, there is a non-negative dilation $($homogeneity structure$)$ on ${\mathbb R}^n$ which can be described in a constructive way such that the first $a_1$ variables are of degree $r_1$, the next $a_2$ variables are of degree $r_2$, etc., up to final $a_m$ variables of degree $r_m$, completely determined $($up to a diffeomorphism$)$ by~$L$, so the eigenvalues and multiplicities of the dilation have invariant meaning.
Moreover, $L\subset\mathfrak{g}^{\le 0}(U,h)$ and
$$L^j\subset \bigoplus_{i=-k}^{-j}\mathfrak{g}^i(U,h) .$$
Actually, there is a finite dimensional commutative subalgebra $\mathfrak{{g}}^{0}_0(U,h)$ in $\mathfrak{{g}}^{0}(U,h)$ such that
\begin{gather}\label{solv} L \subset \mathfrak{g}^{<0}(U,h)\oplus \mathfrak{g}_0^{0}(U,h) .
\end{gather}
\end{Theorem}
\begin{proof} We shall proceed inductively with respect to $n$. The theorem is true for $n=1$ (see Example \ref{ex3}). Suppose it is true for $n-1$, $n\ge 2$. Let $\big(y^1,\dots,y^n\big)$ be variables described by Proposition \ref{p1}, let $h$ be a non-negative dilation (homogeneity structure) on ${\mathbb R}^n$ such that the first $a_1$ variables $y^1,\dots,y^{a_1}$ are of degree $r_1$, the next $a_2$ variables are of degree $r_2$, etc., up to final $a_m$ variables of degree $r_{m}$. Forgetting the last $a_m$ variables (i.e. passing to the manifold ${\mathbb R}^n/{\mathcal F}^k\subset {\mathbb R}^{b_{m-1}}$ of leaves of the foliation ${\mathcal F}^k$), we obtain a non-negative dilation $\bar{h}$ of ${\mathbb R}^{b_{m-1}}$ and Lie algebra $\bar{L}$ described in item~5 of Proposition~\ref{p1}. By the inductive assumption,
\begin{gather*}
\bar{L} \subset \mathfrak{g}^{<0}\big(\bar{U},\bar{h}\big)\oplus \mathfrak{g}_0^{0}\big(\bar{U},\bar{h}\big)
\end{gather*}
and
\[ \bar{L}^j\subset \bigoplus_{i=-k+1}^{-j}\mathfrak{g}^i\big(\bar{U},\bar{h}\big) ,\qquad j=1,\dots,k-1 .\]
Now, in view of Proposition \ref{p1}, every element of $L$ is of the form
\begin{gather}\label{form4} X=\sum_{j=b_{m-1}+1}^n\left(\sum_{i=b_{m-1}+1}^nf_i^j\big(y^1,\dots,y^{b_{m-1}}\big)y^i +g^j\big(y^1,\dots,y^{b_{m-1}}\big)\right)\partial_{y^j}
+\bar{X} ,
\end{gather}
while elements of $L^1$ are of the form
\begin{gather*} X=\sum_{j=b_{m-1}+1}^n u^j\big(y^1,\dots,y^{b_{m-1}}\big)\partial_{y^j}.
\end{gather*}
It suffices to show that functions $f_i^j\big(y^1,\dots,y^{b_{m-1}}\big)$ depend of coordinates of degree 0 only, so they are of degree zero, $g^j\big(y^1,\dots,y^{b_{m-1}}\big)$ and $u^j\big(y^1,\dots,y^{b_{m-1}}\big)$ are of degree $r_m$ and $r_m-1$ respectively.

Let us start with $L^1$. By definition it is a nilpotent Lie algebra. There are two possibilities: $r_1>0$, i.e., foliation ${\mathcal F}^1$ is trivial in a sense that it has only one leaf equal $U$, or $r_1=0$ so the foliation ${\mathcal F}^1$ is nontrivial. In the first case there are no variables of degree~$0$ and $L^1$ is a~transitive, nilpotent Lie algebra. Therefore we can use Theorem \ref{nil} according to which elements of $L^1$ are of degree at most $-1$, so functions $u^j\big(y^1,\dots,y^{b_{m-1}}\big)$ are of degree at most~\mbox{$r_m-1$}. In the second case $L^1$ is a transitive, nilpotent Lie algebra on each leaf of the foliation ${\mathcal F}^1$, i.e., $u_j\big(y^1,\dots,y^{b_{m-1}}\big)$ is of degree $r_m-1$ for each fixed $y^1=\alpha_1,\dots,y^{a_1}=\alpha_{a_1}$. But $y^1,\dots,y^{a_1}$ are of degree $0$ in this case, so $u_j\big(y^1,\dots,y^{b_{m-1}}\big)$ are of degree $r_m-1$ as a whole. Since variables $\big(y^{b_{m-1}+1},\ldots, y^n\big)$ are of degree $r_m$, this means that
\begin{gather*}
L^1\subset\mathfrak{g}^{<0}(U,h) .
\end{gather*}
Moreover, as $L$ is not nilpotent and $L^1\ne\{0\}$, elements of $L^k$ are of the form
\begin{gather}\label{form5} X=\sum_{j=b_{m-1}+1}^n u_j\big(y^1,\dots,y^{a_1}\big)\partial_{y^j} ,
\end{gather}
where $y^1,\dots,y^{a_1}$ are coordinates of degree 0 (or $u^j$ are constants when there are no coordinates of degree zero).

Take now $X\in L$ as given in (\ref{form4}). Since
\[ [\partial_{y^s},X]=\sum_{j=b_{m-1}+1}^nf_s^j\big(y^1,\dots,y^{b_{m-1}}\big)\partial_{y^j}\in L^k\]
for $s=b_{m-1}+1,\dots,n$, we get according to (\ref{form5}) that $f_s^j\big(y^1,\dots,y^{b_{m-1}}\big)$
are homogeneous of degree 0, i.e.,
\begin{gather*}
f_s^j=f_s^j\big(y^1,\dots,y^{a_1}\big) \qquad \text{for} \quad s,j=b_{m-1}+1,\dots,n .\end{gather*}

What remains now is to prove that $g^j\big(y^1,\dots,y^{b_{m-1}}\big)$ are polynomial of degree $\le r_m$. We will do it inductively, proving that, for each $s=1,\dots,m$, the derivatives
$\partial_{y^r}g_j$ where $r=b_{s-1}+1,\dots,b_s$, are polynomial of degree $\le r_m-r_s$. This is trivially true for $s=m$,
as $g_j$ do not depend on variables $y^{b_{m-1}+1},\dots, y^n$. Suppose it is true for $s>1$, we will show that it true for $s-1$. Since $[\partial_{y^r},X]\in L^{r_{s-1}}$ for $r=b_{s-2}+1,\dots,b_{s-1}$ (recall that such $\partial_{y^r}\in L^{r_{s-1}} \operatorname{mod} L^{r_{s-1}+1})$, this implies that
\[\partial_{y^r}g_j\big(y^1,\dots,y^{b_{m-1}}\big)\]
are polynomials of degree $\le r_m-r_{s-1}$.
Now it is easy to see that, as $\partial_{y^r}g_j\big(y^1,\dots,y^{b_{m-1}}\big)$ is polynomial of degree $r_m-\deg(\partial_{y^r})$ for each $r$, the function $g_j\big(y^1,\dots,y^{b_{m-1}}\big)$ itself is polynomial of degree $\le r_m$.

Thus we have shown that $L\subset\mathfrak{g}^{\le 0}(U,h)$. Let $\mathfrak{g}^{0}_0(U,h)$ be the projection of $L$ onto $\mathfrak{g}^{0}(U,h)$. Then we have (\ref{solv}) and $\mathfrak{g}^{0}_0(U,h)$ is commutative.
Indeed, it follows from $[L,L]=L^1\subset\mathfrak{g}^{< 0}(U,h)$ and that $\mathfrak{g}^{0}(U,h)$ is a Lie algebra of vector fields. Since $L^1\subset \bigoplus_{i=-k}^{-1}\mathfrak{g}^i(U,h)$ and $L^j=\big[L^1,\big[L^1\big[\dots,\big[L^1,L^1\big]\dots\big]\big]\big]$ for $j>1$, we get
\[L^j\subset \bigoplus_{i=-k}^{-j}\mathfrak{g}^i(U,h) .\tag*{\qed}\]
\renewcommand{\qed}{}
\end{proof}

Note that finite-dimensionality was assumed to achieve two properties: the fact that $[L,L]$ is nilpotent and that~${\mathcal F}^i$ are regular foliations. Our result remains therefore valid for transitive solvable Lie algebras of vector fields (non-necessarily finite-dimensional) if we assure these two properties.
\begin{Example} Take the Lie algebra from Example~\ref{1.1}. It is infinite-dimensional.
\end{Example}

The second required property is valid for instance for Lie algebras of analytic vector fields, as we have in this case a local transitive action of local diffeomorphism which can be written as convergent series of Lie bracket operations, so the distributions spanned by~$L^i$ are regular and integrable (cf.~\cite{nagano66}). In the smooth category, however, solvable Lie algebras of vector fields need not to define the corresponding foliations.
\begin{Example}
Let $L$ be the Lie algebra of vector fields spanned in ${\mathbb R}^2$ by vector fields $\langle f(y)\partial_x{,}\partial_y\rangle,\!$ where $f(y)$ can be taken as an arbitrary smooth function on ${\mathbb R}^2$ which is 1 for $y>0$. The algebra~$L$ is clearly solvable, but the distribution generated by $[L,L]$ is 0 for $y\ge 0$ and spanned by~$\partial_x$ for~$y<0$. It is not integrable.
\end{Example}
Thus we get.
\begin{Theorem}Theorem~{\rm \ref{main}} remains valid for transitive solvable Lie algebras $L$ of $($real$)$ analytic vector fields such that $[L,L]$ is nilpotent.
\end{Theorem}

\section{Lie's conjecture}
Sophus Lie conjectured (see~\cite{Lie3}), that for any finite-dimensional complex transitive Lie algebra~$L$ of vector fields in~${\mathbb R}^n$ one can find coordinates $y^1,\dots,y^n$ in which all coefficients
of vector fields from $L$ lie in the algebra generated by the~$y^i$ and the exponentials $\exp\big(\lambda y^i\big)$. For instance, we know already that nilpotent transitive Lie algebras of vector fields are polynomial in certain coordinates. This is no longer true for solvable algebras.

Draisma in \cite{Draisma02a} (see also \cite{Draisma02}) proved a version of the Lie's conjecture for a class of Lie algebras of vector fields including solvable transitive algebras. Being a nice result this is however not the exact version of the Lie's conjecture, as it speaks only about the existence of certain `good' realizations, while in the Lie's conjecture the realization is already given and we look only for ``good'' appropriate coordinates. We propose a proof of the exact and constructive version of the Lie's conjecture for solvable algebras based on our description of the structure of transitive solvable algebras of vector fields.

Since finite-dimensional solvable and transitive Lie algebras of vector fields $L$ are dilational, we have
coordinates $x^1,\dots,x^p$ of zero degree (weight) and coordinates of positive degree (weight) $z^1_{w_1},\dots,z^q_{w_q}$, where $w_i>0$ indicates the weight, such that the coefficients of vector fields from $L$
are polynomials with coefficients in functions depending on $x$:
\[ f(x,y)=\sum_{|\alpha|\le w}f_\alpha(x)\big(z_{w_1}^1\big)^{\alpha_1}\cdots \big(z_{w_q}^q\big)^{\alpha_q} .\]
As $\partial_{x^1},\dots,\partial_{x^p}$ belong to $L$ $\operatorname{mod} L^1$, for any $X\in L$ we can consider the iterated brackets
$\operatorname{ad}^s(\partial_{x^i})(X)$. But $L$ is finite-dimensional so a finite number of the iterated brackets are linearly dependent. Inductively with respect to $w=|\alpha|$ we can show therefore that a finite number of
$\partial^s_{x^i}(f_\alpha(x))$ is linearly dependent. This leaves to a system of linear equations
\[ \sum_{s,|\alpha|=w} c_s^\alpha\frac{\partial^s f_\alpha}{(\partial x^i)^s}(x)=0 .\]
Its complex solutions $f_\alpha(x)$ are in the associative algebra generated by exponentials $\exp\big(\lambda x^i\big)$.
Thus we have proved the following.
\begin{Theorem}
Finite-dimensional solvable and transitive Lie algebras of vector field verify the Lie's conjecture.
\end{Theorem}

\begin{Example}Consider the Lie algebra of vector fields (\ref{ida}), where ${\mathbb A}=\{ f(x)\colon f\in C^\infty({\mathbb R})\}$, and its subalgebra $L$ spanned by vector fields with some exponential coefficients with respect to~$x$:
\begin{gather*}
\langle \partial_x, y\partial_y, z\partial_z, u\partial_u, \sin(x)\partial_u, \cos(x)\partial_u, \mathrm{e}^xy\partial_u,\partial_z, \mathrm{e}^{2x}\partial_u, \partial_u, \\
 \qquad{} \cos(x)z\partial_u, \sin(x)z\partial_u, y^2\partial_u, y\partial_z, \partial_y, \sin(x)y\partial_u, \cos(x)y\partial_u,y\partial_u\rangle .
\end{gather*}
It is indeed a finite-dimensional and solvable Lie algebra, since $L^1=[L,L]\subset L$:
\begin{gather*}
L^1=\big\langle \sin(x)\partial_u, \cos(x)\partial_u, \mathrm{e}^xy\partial_u, \partial_u, \partial_z, \mathrm{e}^{2x}\partial_u,\\
\hphantom{L^1=\big\langle}{} \cos(x)z\partial_u, \sin(x)z\partial_u, y^2\partial_u, y\partial_z, \partial_y, \sin(x)y\partial_u, \cos(x)y\partial_u,y\partial_u\big\rangle .
\end{gather*}
We have in turn
\begin{gather*}
L^2=\langle \sin(x)\partial_u, \cos(x)\partial_u, \partial_y, \sin(x)y\partial_u, \cos(x)y\partial_u,y\partial_u, \partial_u, \partial_z\rangle ,\\
L^3=\langle \sin(x)\partial_u, \cos(x)\partial_u,\partial_u\rangle
\end{gather*}
and
\[ L^4=\langle 0\rangle .\]
\end{Example}

\section{Concluding remarks}
Dilations, or more generally homogeneity structures which can be viewed as generalizations of vector bundles \cite{Bruce:16, Grabowski:2012}, provide a powerful method for very simple construction of solvable Lie algebras of vector fields. We just take vector fields of negative degrees with respect to the homogeneity structure and add a commutative Lie algebra of vector fields of degree $0$. We have shown that any finite-dimensional solvable and transitive Lie algebra of vector fields $L$ (or Lie algebra $L$ of analytical vector fields such that $[L,L]$ is nilpotent) induces locally the corresponding homogeneity structure which contains $L$ as a Lie subalgebra of vector fields of degrees $\le 0$. This is a natural generalization of the corresponding description of nilpotent transitive Lie algebras of vector fields \cite{JG90}. However, we could not just use directly \cite{JG90}, as $[L,L]$ is in general no longer transitive.

We can view the Lie algebras of vector fields of degrees $\le 0$ (resp.~$<0$) with respect to a fixed homogeneity structure as `completions' of the given solvable (nilpotent) transitive Lie algebras of vector fields, and study their algebraic properties associated with such a description (e.g., Tanaka prolongations).

\subsection*{Acknowledgements}
We thank the anonymous referees for their careful reading of our manuscript and their insightful comments.
Following their suggestions, we included several improvements in the manuscript.
Research founded by the Polish National Science Centre grant under the contract number 2016/22/M/ST1/00542.

\pdfbookmark[1]{References}{ref}
\LastPageEnding

\end{document}